\documentclass[a4paper,fleqn]{cas-sc}

\usepackage[numbers]{natbib}
\usepackage{amsmath,amssymb,amsthm,mathtools}
\usepackage{algorithm}
\usepackage{algpseudocode}
\usepackage{booktabs}
\usepackage{graphicx}
\usepackage{microtype}
\usepackage{flushend}

\RenewDocumentCommand{\printorcid}{}{}

\newtheorem{theorem}{Theorem}[section]
\newtheorem{proposition}[theorem]{Proposition}
\newtheorem{lemma}[theorem]{Lemma}

\theoremstyle{remark}

\newcommand{\R}{\mathbb{R}}
\newcommand{\sym}{\operatorname{sym}}
\newcommand{\prox}{\operatorname{prox}}
\newcommand{\argmax}{\operatorname*{arg\,max}}

\begin{document}
\let\WriteBookmarks\relax
\def\floatpagepagefraction{1}
\def\textpagefraction{.001}

\shorttitle{Coordinate Proximal Predictor--Corrector for Monotone AVEs}
\shortauthors{H. Wang and Y. Xia}

\title[mode=title]{A Line-Search-Free Coordinate Proximal Predictor--Corrector Method for Monotone Absolute Value Equations}

\author[1]{Haotian Wang}
\author[1]{Yong Xia}

\cormark[1]
\ead{yxia@buaa.edu.cn}
\cortext[1]{Corresponding author.}
\affiliation[1]{organization={School of Mathematical Sciences, Beihang University},
                addressline={No. 37 Xueyuan Road},
                city={Beijing}, postcode={100191}, country={China}}

\begin{abstract}
	We consider the absolute value equation (AVE) $Ax-|x|=b$ under residual
	monotonicity.  We characterize this property through the symmetric part of the
	coefficient matrix, thereby allowing nonsymmetry, and propose a coordinate proximal
	predictor--corrector (CPPC) method based on an AVE-specific forward--proximal
	decomposition.  An exact scalar proximal update on the largest proximal-residual
	coordinate generates a predictor point, while a positive-alignment identity certifies
	the ensuing full-residual separating-hyperplane correction without backtracking.
	With the current matrix product cached, each iteration requires one new full
	matrix--vector product.  For a nonempty solution set, we prove Fej\'er monotonicity,
	whole-sequence convergence, and an $O(K^{-1/2})$ best-iterate residual bound.  A
	positive monotonicity margin further ensures unique solvability for every right-hand
	side and global linear convergence.  Numerical results identify regimes in which the
	reduced per-iteration work yields shorter solution times.
\end{abstract}

\begin{keywords}
	absolute value equation \sep monotone equation \sep coordinate proximal residual
	\sep predictor--corrector method \sep Fej\'er monotonicity
	\MSC{65H10 \sep 65K05 \sep 90C30}
\end{keywords}

\maketitle

\section{Introduction}

The absolute value equation (AVE)
\begin{equation}\label{eq:ave}
    0=\Phi(x):=Ax-|x|-b,
\end{equation}
where $A\in\R^{n\times n}$ and $b\in\R^n$, is a piecewise linear and generally
nonsmooth system.  It is closely related to linear complementarity problems and
arises in several mathematical programming reformulations
\cite{MangasarianMeyer2006,Mezzadri2020}.  Numerical methods span several families.
Optimization-based approaches include concave-minimization, proximal, and monotone
block coordinate descent methods \cite{Mangasarian2007,ShahsavariKetabchi2021,Luo2025}.
Generalized and semismooth Newton schemes are designed to achieve fast local or global convergence \cite{Mangasarian2009,Caccetta2011}, whereas Picard fixed-point, SOR-like, and
matrix-splitting iterations exploit linear-algebraic structure
\cite{Salkuyeh2014,GuoWuLi2019,LvMiao2024}.  Recent developments include
maximum-based iterations for generalized AVEs \cite{WuHanLi2024} and generalized
Gauss--Seidel iterations under $M$- and $H$-matrix structure \cite{Guo2025}.
Douglas--Rachford splitting has
also been developed for large-scale sparse AVEs \cite{ChenYuHan2023}.  Broader
accounts are given in \cite{Moosaei2015,Hladik2026}.

This paper is concerned with projection-type first-order methods.  The spectral
gradient projection (SGP) method of Yu et al.\ \cite{Yu2009SGP} solves monotone
nonlinear equations by combining a scalar spectral direction, a backtracking search,
and a separating-hyperplane correction.  Yu, Li and Yuan \cite{Yu2021MMSGP}
specialized this framework to \eqref{eq:ave} and replaced the scalar spectral factor
by a diagonal multivariate scaling.  Hua and Ma \cite{HuaMa2022} subsequently
introduced a relaxed spectral difference and a modified line search.  We refer to
these AVE variants as MMSGP and MMSGPs, respectively.  Their trial points have the
generic form
\[
    z^k=x^k+\alpha_kd^k,
\]
where $d^k$ is a scalar- or diagonal-scaled residual direction and $\alpha_k$ is obtained
by backtracking.  The search enforces
\[
    \langle\Phi(z^k),x^k-z^k\rangle>0,
\]
so that $\Phi(z^k)$ defines a separating-hyperplane correction; rejected trials require
further full-residual evaluations.  A line-search-free variant therefore needs an
analytic alignment certificate.

We propose a coordinate proximal predictor--corrector (CPPC) method that obtains such
a certificate from an exact scalar proximal predictor and retains the full-residual
separating-hyperplane corrector.  Since the predictor changes only one coordinate, its residual can be obtained from the cached matrix product by a single column update and leaves one new matrix--vector product per iteration.  We characterize monotonicity
by $\sym(A)\succeq I$ and, for a nonempty solution set, prove whole-sequence convergence
and an $O(K^{-1/2})$ best-iterate residual bound; a positive monotonicity margin yields
global linear convergence.  The predictor follows the mini-extragradient principle
\cite{LiuXia2025}, while its exact construction and certificate exploit the AVE structure.

The remainder is organized as follows.  Section~\ref{sec:formulation} gives the
monotonicity characterization and the proximal residual.  Section~\ref{sec:method}
states the algorithm, Section~\ref{sec:convergence} gives the convergence and rate analysis,
and Section~\ref{sec:numerical} reports numerical results.

\section{AVE structure and proximal residual}\label{sec:formulation}

All inner products are Euclidean; $I$ is the identity, $e_i$ is the $i$th coordinate
vector, and $\|\cdot\|$ denotes the Euclidean norm or its induced matrix norm.  Let
$\sym(A)=(A+A^\top)/2$, and define the componentwise negative part by
$x^-_i=\max\{-x_i,0\}$.  The solution set of \eqref{eq:ave} is denoted by $S:=\{x\in\R^n:\Phi(x)=0\}$.

\subsection{Monotonicity and structured splitting}

The following result gives the precise monotone regime of the AVE residual.

\begin{proposition}\label{prop:monotonicity}
The mapping $\Phi$ is monotone on $\R^n$ if and only if
\begin{equation}\label{eq:monotonicity-condition}
    \sym(A)\succeq I.
\end{equation}
\end{proposition}

\begin{proof}
Let $z=x-y$.  Since
$\langle |x|-|y|,x-y\rangle\leq\|x-y\|^2$, condition
\eqref{eq:monotonicity-condition} gives
\[
\begin{aligned}
\langle\Phi(x)-\Phi(y),x-y\rangle
=z^\top Az-\langle |x|-|y|,z\rangle
\geq z^\top(\sym(A)-I)z\geq0.
\end{aligned}
\]
Conversely, fix $z\in\R^n$ and choose $y$ so that both $y$ and $y+z$ are
componentwise positive.  Setting $x=y+z$ gives $|x|-|y|=z$.  Monotonicity then
implies $z^\top(\sym(A)-I)z\geq0$.  Since $z$ is arbitrary,
\eqref{eq:monotonicity-condition} follows.
\end{proof}

Condition \eqref{eq:monotonicity-condition} depends only on the symmetric
part of $A$. Hence nonsymmetric matrices are allowed, and monotonicity imposes
no restriction on the skew-symmetric part.

Set
{\small
\begin{equation}\label{eq:Fg}
    F(x):=(A-I)x-b,\qquad
    G(x):=\|x^-\|^2=\sum_{i=1}^n g(x_i),\qquad
    g(t):=(t^-)^2.
\end{equation}
}
Then $g'(t)=t-|t|$ and
\begin{equation}\label{eq:structured-splitting}
    \begin{aligned}
    \Phi(x)=F(x)+\nabla G(x)=(A-I)x-b+\nabla \|x^-\|^2.
    \end{aligned}
\end{equation}
For symmetric $A$, integrating \eqref{eq:structured-splitting} gives the
piecewise-quadratic AVE potential used by Noor et al.\ \cite{Noor2012}. Here the identity is used directly at the operator level and therefore does not require symmetry: $F$ is linear, while $g$ has an exact scalar proximal map.  This permits an exact coordinate predictor whose interaction with the unsplit residual $\Phi$ can be evaluated algebraically.
Throughout the rest of the paper we assume
\begin{equation}\label{eq:base-assumption}
    \sym(A)\succeq I,
    \qquad S\neq\varnothing.
\end{equation}
The nonemptiness condition is separate: monotonicity at the semidefinite boundary
does not ensure a solution for every $b$.

\subsection{Exact coordinate proximal residual}

Choose $0<\rho<1$, $\epsilon>0$ and constants
\begin{equation}\label{eq:gamma}
    \ell_i=\max\{a_{ii}-1,\epsilon\},
    \qquad \gamma_i:=\frac{\rho}{\ell_i},
    \qquad i=1,\ldots,n.
\end{equation}
Condition \eqref{eq:monotonicity-condition} implies $a_{ii}\geq1$. For $\gamma>0$,
\begin{equation}\label{eq:prox-formula}
    \prox_{\gamma g}(z)
    =\begin{cases}
       z, & z\geq0,\\[1mm]
       \displaystyle z/(1+2\gamma), & z<0,
     \end{cases}
\end{equation}
where $\prox_{\gamma h}(z):= \arg\min_{t\in\R}\{h(t)+(2\gamma)^{-1}|t-z|^2\}$.

Define the coordinate proximal residual
{\small
\begin{equation}\label{eq:prox-residual}
    R_i(x):=\gamma_i^{-1}
    \left[x_i-\prox_{\gamma_i g}\bigl(x_i-\gamma_iF_i(x)\bigr)\right],
    \qquad R(x):=(R_i(x))_{i=1}^n.
\end{equation}
}
Thus $\gamma_iR_i(x)$ is precisely the displacement produced by the exact scalar
forward--proximal step.  The residual $R$ is used only to select and construct the
predictor; the corrector below continues to use the original AVE residual $\Phi$.
Writing $q_i(x)=x_i-\gamma_iF_i(x)$, formula \eqref{eq:prox-formula} gives
\begin{equation}\label{eq:residual-explicit}
R_i(x)=
\begin{cases}
    ((A-I)x-b)_i, & q_i(x)\geq0,\\[1mm]
    \displaystyle((A+I)x-b)_i/(1+2\gamma_i), & q_i(x)<0.
\end{cases}
\end{equation}

\begin{proposition}\label{prop:stationarity}
For every $x\in\R^n$, $R(x)=0$ if and only if $\Phi(x)=0$.
\end{proposition}

\begin{proof}
The fixed-point characterization of the proximal map yields
{\small
\[
R_i(x)=0
\iff x_i=\prox_{\gamma_i g}(x_i-\gamma_iF_i(x))
\iff F_i(x)+g'(x_i)=0.
\]
}
Apply \eqref{eq:structured-splitting} coordinatewise.
\end{proof}

\section{The line-search-free predictor--corrector method}\label{sec:method}

\subsection{The CPPC iteration}

CPPC has two deliberately different stages: a local exact predictor and a global residual
corrector.  At $x^k$, choose
\begin{equation}\label{eq:greedy-index}
    i_k\in\argmax_{1\leq i\leq n}|R_i(x^k)|
\end{equation}
and change only this coordinate:
\begin{equation}\label{eq:predictor}
    y_j^k=x_j^k\quad(j\neq i_k),
    \qquad
    y_{i_k}^k=\prox_{\gamma_{i_k}g}
    \bigl(x_{i_k}^k-\gamma_{i_k}F_{i_k}(x^k)\bigr).
\end{equation}
Let $d^k:=x^k-y^k$ and $v^k:=\Phi(y^k)$.
The vector $v^k$ is the full AVE residual, not a coordinate approximation.  If
$d^k\neq0$, set
\begin{equation}\label{eq:corrector}
    \lambda_k:=\frac{\langle v^k,d^k\rangle}{\|v^k\|^2},
    \qquad
    x^{k+1}:=x^k-\lambda_kv^k.
\end{equation}
Lemma~\ref{lem:alignment} below proves algebraically that $v^k\neq0$ and
$\lambda_k>0$ at every nonterminal iteration; no acceptance search is required.

\begin{algorithm}[t]
\caption{Line-search-free coordinate proximal predictor--corrector method (CPPC)}
\label{alg:cppc}
\begin{algorithmic}[1]
\Require $A,b,x^0$, $0<\rho<1$, and $\ell_i$ satisfying \eqref{eq:gamma}
\State Set $\gamma_i=\rho/\ell_i$, $i=1,\ldots,n$, and $u^0=Ax^0$
\For{$k=0,1,2,\ldots$}
    \State Evaluate $R(x^k)$ from $u^k$ and choose $i_k$ by \eqref{eq:greedy-index}
    \If{$R_{i_k}(x^k)=0$}
        \State \Return $x^k$
    \EndIf
    \State Form $y^k$ by \eqref{eq:predictor}; set $d^k=x^k-y^k$
    \State Set $v^k=u^k+(y_{i_k}^k-x_{i_k}^k)Ae_{i_k}-|y^k|-b$
    \State Set $\lambda_k=\langle v^k,d^k\rangle/\|v^k\|^2$
    \State Set $x^{k+1}=x^k-\lambda_kv^k$ and $u^{k+1}=Ax^{k+1}$
\EndFor
\end{algorithmic}
\end{algorithm}

\subsection{Relation to spectral projection iterations}

The AVE-specific predictor treats the absolute-value term exactly rather than through a
coordinate Lipschitz bound.  Indeed, under \eqref{eq:monotonicity-condition},
\begin{equation}\label{eq:coordinate-constants}
\begin{aligned}
 |F_i(x+te_i)-F_i(x)|&=(a_{ii}-1)|t|,\\
 |\Phi_i(x+te_i)-\Phi_i(x)|&\leq(a_{ii}+1)|t|.
\end{aligned}
\end{equation}
Thus the forward part of the proximal predictor uses the smaller constant $a_{ii}-1$,
while \eqref{eq:prox-formula} resolves the remaining scalar nonlinearity without an
approximation.

For the unconstrained AVE, all four methods use the full-residual correction
\begin{equation}\label{eq:spectral-correction}
    x^{k+1}=x^k-
    \frac{\langle\Phi(z^k),x^k-z^k\rangle}{\|\Phi(z^k)\|^2}\Phi(z^k).
\end{equation}
SGP takes $d^k=-\theta_k\Phi(x^k)$, whereas MMSGP and MMSGPs use
$d^k=-D_k\Phi(x^k)$ with a safeguarded diagonal matrix $D_k$, and determine
$z^k=x^k+\alpha_kd^k$ by backtracking.  CPPC instead takes $z^k=y^k$ and
Lemma~\ref{lem:alignment} certifies that the numerator is positive.  Thus CPPC removes the line search without changing the full-residual correction geometry.

With $Ax^k$ cached,
\[
    Ay^k=Ax^k+(y_{i_k}^k-x_{i_k}^k)Ae_{i_k},
\]
so $v^k$ is obtained by a column update, and the only new global matrix--vector
product is $Ax^{k+1}$.  By contrast, the spectral methods require products along the
search direction and at the corrected iterate, together with full-vector residual
tests during backtracking.  CPPC therefore replaces one global product and the line
search with a coordinate scan and a column update.  The benefit depends on the
problem structure and iteration count.

\section{Convergence and rate analysis}\label{sec:convergence}

\subsection{Global convergence and best-iterate rate}

The predictor displacement satisfies
\begin{equation}\label{eq:displacement}
    d_{i_k}^k=\gamma_{i_k}R_{i_k}(x^k),
    \qquad
    \|d^k\|=\gamma_{i_k}\|R(x^k)\|_\infty.
\end{equation}
The following identity is the acceptance certificate that replaces backtracking.

\begin{lemma}[Exact alignment]\label{lem:alignment}
At every nonterminal iteration,
$\langle v^k,d^k\rangle>0$.
\end{lemma}

\begin{proof}
The optimality condition in \eqref{eq:predictor} is
\[
    \frac{x_{i_k}^k-y_{i_k}^k}{\gamma_{i_k}}-F_{i_k}(x^k)
    =g'(y_{i_k}^k).
\]
Since $d^k$ is supported on $i_k$ and
$F_{i_k}(y^k)-F_{i_k}(x^k)=-(a_{i_ki_k}-1)d_{i_k}^k$, we obtain
\[
\langle v^k,d^k\rangle
=\left(\frac{1}{\gamma_{i_k}}-(a_{i_ki_k}-1)\right)\|d^k\|^2
\geq\frac{1-\rho}{\gamma_{i_k}}\|d^k\|^2.
\]
Here the inequality follows from $\ell_{i_k}\geq a_{i_ki_k}-1$ and
$\gamma_{i_k}=\rho/\ell_{i_k}$.  At a nonterminal iteration,
\eqref{eq:displacement} gives $d^k\neq0$; hence the right-hand side is positive.
\end{proof}

For $x^*\in S$, monotonicity gives $\langle v^k,y^k-x^*\rangle\geq0$. Thus $S$ lies in the halfspace
\[
    \mathcal H_k:=\{z:\langle v^k,z-y^k\rangle\leq0\},
\]
whereas Lemma~\ref{lem:alignment} places $x^k$ strictly outside $\mathcal H_k$.
Consequently, the exact coordinate predictor generates a valid separating hyperplane for
the full AVE residual, and the corrector \eqref{eq:corrector} projects $x^k$ onto its boundary.

\begin{lemma}[Fej\'er decrease]\label{lem:fejer}
For every $x^*\in S$,
\begin{equation}\label{eq:fejer}
    \|x^{k+1}-x^*\|^2
    \leq\|x^k-x^*\|^2-
    \frac{\langle v^k,d^k\rangle^2}{\|v^k\|^2}.
\end{equation}
\end{lemma}

\begin{proof}
Since $x^k=y^k+d^k$,
$\langle v^k,x^k-x^*\rangle\geq\langle v^k,d^k\rangle$.
Expanding \eqref{eq:corrector} and substituting the definition of $\lambda_k$
gives \eqref{eq:fejer}.
\end{proof}

We next relate the proximal residual to $\Phi$.  Let
\[
    p_i(x):=\prox_{\gamma_i g}(x_i-\gamma_iF_i(x))
           =x_i-\gamma_iR_i(x).
\]
Since $g'$ is $2$-Lipschitz,
\[
    |\Phi_i(x)-R_i(x)|
    =|g'(x_i)-g'(p_i(x))|
    \leq2\gamma_i|R_i(x)|.
\]
Define
{\small
\begin{equation}\label{eq:constants}
\begin{split}
    C_R:=\left(\sum_{i=1}^n(1+2\gamma_i)^2\right)^{1/2},\quad
    C_S:=\max_i\gamma_i(\|Ae_i\|+1),\qquad
     C_\Phi:=C_R+C_S,
    \qquad \gamma_*:=\min_i\gamma_i.
\end{split}
\end{equation}
}
Then, with $r_k:=\|R(x^k)\|_\infty$,
{\small
\begin{equation}\label{eq:residual-bounds}
    \|\Phi(x^k)\|\leq C_Rr_k,
    \qquad
    \|v^k\|\leq C_\Phi r_k,
    \qquad
    \langle v^k,d^k\rangle\geq(1-\rho)\gamma_*r_k^2.
\end{equation}
}
Combining \eqref{eq:fejer} and \eqref{eq:residual-bounds} yields
\begin{equation}\label{eq:fundamental}
    \|x^{k+1}-x^*\|^2
    \leq\|x^k-x^*\|^2-\alpha r_k^2,
    \qquad
    \alpha:=\frac{(1-\rho)^2\gamma_*^2}{C_\Phi^2}.
\end{equation}

\begin{theorem}[Global convergence]\label{thm:global}
Under \eqref{eq:base-assumption} and \eqref{eq:gamma}, Algorithm~\ref{alg:cppc}
is well defined and either terminates at a point in $S$ or generates a sequence
converging to a point in $S$.  In the latter case,
\[
    \sum_{k=0}^{\infty}\|R(x^k)\|_\infty^2<\infty,
    \qquad R(x^k)\to0,
    \qquad \Phi(x^k)\to0.
\]
\end{theorem}

\begin{proof}
At termination, the stopping test and \eqref{eq:greedy-index} give $R(x^k)=0$, hence $x^k\in S$ by
Proposition~\ref{prop:stationarity}.  Otherwise, Lemma~\ref{lem:alignment} guarantees
a positive denominator and numerator in every
nonterminal correction.  Summing \eqref{eq:fundamental} gives
$\sum_kr_k^2<\infty$, hence $r_k\to0$ and, by \eqref{eq:residual-bounds},
$\Phi(x^k)\to0$.  Lemma~\ref{lem:fejer} makes $\{x^k\}$ Fej\'er monotone with
respect to $S$, hence bounded.  Every cluster point belongs to $S$ by continuity of
$\Phi$.  Fej\'er monotonicity then implies convergence of the whole sequence to such a
cluster point.
\end{proof}

\begin{theorem}[Best-iterate residual rate]\label{thm:complexity}
For every $x^*\in S$ and index $K\geq0$,
\begin{align}
    &\min_{0\leq k\leq K}\|R(x^k)\|_\infty
    \leq
    \frac{C_\Phi\|x^0-x^*\|}
         {(1-\rho)\gamma_*\sqrt{K+1}},\label{eq:best-R}\\
    &\min_{0\leq k\leq K}\|Ax^k-|x^k|-b\|
    \leq
    \frac{C_RC_\Phi\|x^0-x^*\|}
         {(1-\rho)\gamma_*\sqrt{K+1}}.\label{eq:best-Phi}
\end{align}
\end{theorem}

\begin{proof}
Telescoping \eqref{eq:fundamental} gives
$(K+1)\min_{0\leq k\leq K}r_k^2\leq\|x^0-x^*\|^2/\alpha$.
This proves \eqref{eq:best-R}; the first inequality in
\eqref{eq:residual-bounds} gives \eqref{eq:best-Phi}.
\end{proof}

Theorem~\ref{thm:complexity} makes the residual complexity explicit, whereas the original
analyses of the compared spectral projection methods establish qualitative
convergence without such a bound.

\subsection{Linear rate under strong monotonicity}

Assume that for some $m>0$,
\begin{equation}\label{eq:strong-condition}
    \sym(A)\succeq(1+m)I.
\end{equation}
Then $\Phi$ is $m$-strongly monotone.
Moreover, \eqref{eq:strong-condition} implies
$\sigma_{\min}(A)\geq1+m>1$; hence \eqref{eq:ave} has a unique solution for every
$b\in\R^n$ by \cite[Proposition~3(i)]{MangasarianMeyer2006}.

\begin{theorem}[Linear convergence]\label{thm:linear}
Suppose \eqref{eq:strong-condition} and \eqref{eq:gamma} hold, and let $x^*$ be the
unique solution.  With
\begin{equation}\label{eq:linear-factor}
    \eta:=(\alpha m^2)/C_R^2\in(0,1),
\end{equation}
the CPPC iterates satisfy
\begin{equation}\label{eq:linear-rate}
    \|x^{k+1}-x^*\|^2
    \leq(1-\eta)\|x^k-x^*\|^2.
\end{equation}
Consequently, $x^k\to x^*$ globally and the AVE residual converges R-linearly.
\end{theorem}

\begin{proof}
Strong monotonicity gives $m\|x^k-x^*\|\leq\|\Phi(x^k)\|\leq C_Rr_k$. Substitution in \eqref{eq:fundamental} proves \eqref{eq:linear-rate}.  Furthermore,
$\gamma_*m\leq\rho$ follows from $\ell_i\geq a_{ii}-1\geq m$; thus, by the
definitions of $\alpha$ and $\eta$,
$0<\eta\leq\rho^2(1-\rho)^2/(C_\Phi^2C_R^2)<1$.  Finally, $\Phi$ is Lipschitz continuous with
constant at most $\|A\|+1$, so the distance estimate implies R-linear residual
convergence.
\end{proof}

\section{Numerical experiments}\label{sec:numerical}

\paragraph{Setup}
We compare CPPC with SGP, MMSGP and MMSGPs in Python 3.11.5 using NumPy 1.24.3
with MKL on a 2.60 GHz Intel Core i7-10750H and 16 GB RAM.  BLAS is restricted
to one thread and all matrices are stored densely.  Each deterministic instance is
run once, with the method order rotated between instances.  We stop when $\frac{\|\Phi(x^k)\|}{\max\{1,\|b\|\}}\leq10^{-6}$ or after $10^4$ iterations.  For CPPC, $\rho=0.75$ and $\ell_i=a_{ii}-1$; the spectral methods use $\beta=0.5$, $\sigma=0.01$ and
$r=0.1$.  MMSGP and MMSGPs additionally use $\tau=10^{-3}$ and $\delta=10^{-2}$ whose numerical values were not specified in the original experiments, and the MMSGPs relaxation is $1.6$.  The seed is 2026, and all methods receive the same $(A,b,x^0)$.

The test matrices are constructed from
\begin{equation}\label{eq:test-matrix}
    A=(1+m)I+L_w+\varepsilon L_{\rm tail}+\kappa K_w,
\end{equation}
where $L_w$ is the band-graph Laplacian with edge weight $1/w$ for offsets
$1,\ldots,w$, $K_w^\top=-K_w$ has upper-band entries $1/(2w)$, and the
distance-$j$ edge weight of $L_{\rm tail}$ is proportional to
$[1+(j/w)^2]^{-1}$ and normalized to total weight $1/2$.  Consequently,
$\sym(A)=(1+m)I+L_w+\varepsilon L_{\rm tail}\succeq(1+m)I$.
In every test, an exact solution $x^*$ is prescribed, $b=Ax^*-|x^*|$, and
$x^0=0$.

\paragraph{Effect of dimension.}
We first set $m=0.05$, $w=5$, $\varepsilon=0.01$ and $\kappa=0.5$.
The support of $x^*$ is one block containing approximately $0.05n$ coordinates,
with alternating values $\pm3$.  Thus \eqref{eq:test-matrix} is genuinely dense,
while the solution and dominant couplings remain local.  Table~\ref{tab:dimension}
reports wall-clock time, iterations and full matrix--vector products (MV).

\begin{table}[H]
\caption{Dimension test: time (s), iterations and MV; the smallest time and MV in each row are bold.}
\label{tab:dimension}
\centering
\footnotesize
\setlength{\tabcolsep}{2.5pt}
\begin{tabular}{r*{4}{rrr}}
\toprule
$n$ & \multicolumn{3}{c}{CPPC} & \multicolumn{3}{c}{SGP}
    & \multicolumn{3}{c}{MMSGP} & \multicolumn{3}{c}{MMSGPs}\\
\cmidrule(lr){2-4}\cmidrule(lr){5-7}\cmidrule(lr){8-10}\cmidrule(lr){11-13}
& Time & Iter. & MV & Time & Iter. & MV & Time & Iter. & MV & Time & Iter. & MV\\
\midrule
10   & 0.0147 & 367  & \textbf{368}  & \textbf{0.0124} & 214 & 429
     & 0.0196 & 204 & 409  & 0.0172 & 185 & 371\\
50   & 0.0128 & 269  & 270  & \textbf{0.0049} & 88  & 177
     & 0.0251 & 250 & 501  & 0.0060 & 55  & \textbf{111}\\
100  & \textbf{0.0071} & 166 & \textbf{167} & 0.0094 & 162 & 325
     & 0.0267 & 294 & 589  & 0.0276 & 188 & 377\\
500  & 0.0960 & 1093 & 1094 & \textbf{0.0516} & 372 & \textbf{745}
     & 0.1466 & 708 & 1417 & 0.1058 & 454 & 909\\
1000 & \textbf{0.0698} & 195 & \textbf{196} & 0.4425 & 525 & 1051
     & 0.4444 & 764 & 1529 & 0.3562 & 676 & 1353\\
2000 & \textbf{0.7708} & 392 & \textbf{393} & 2.5868 & 566 & 1133
     & 4.3680 & 819 & 1639 & 2.4501 & 581 & 1163\\
3000 & \textbf{2.4756} & 589 & \textbf{590} & 4.7610 & 546 & 1093
     & 8.4454 & 836 & 1673 & 5.0632 & 587 & 1175\\
\bottomrule
\end{tabular}
\end{table}

The ranking is mixed for $n\leq500$, where SGP is often faster, whereas CPPC is
consistently fastest for $n\geq1000$ and also attains the smallest MV count.
Thus the large-$n$ advantage reflects both competitive iteration counts and the
reduction from two spectral products per iteration to one product and a column update.

\paragraph{Residual localization and coupling width.}
With $n=1000$, $m=0.05$ and $\varepsilon=\kappa=0$, we vary one feature at a time:
the support fraction in $\{0.01,0.05,0.1,0.2,0.5,1\}$, a contiguous or uniformly
dispersed support, and $w\in\{1,5,20,100\}$.  The respective baseline values are
$0.05$, contiguous support and $w=1$.
Figure~\ref{fig:structure} plots the runtime ratio relative to CPPC, so values above
one favor CPPC.

\begin{figure}[pos=!t]
\centering
\includegraphics[width=0.8\linewidth]{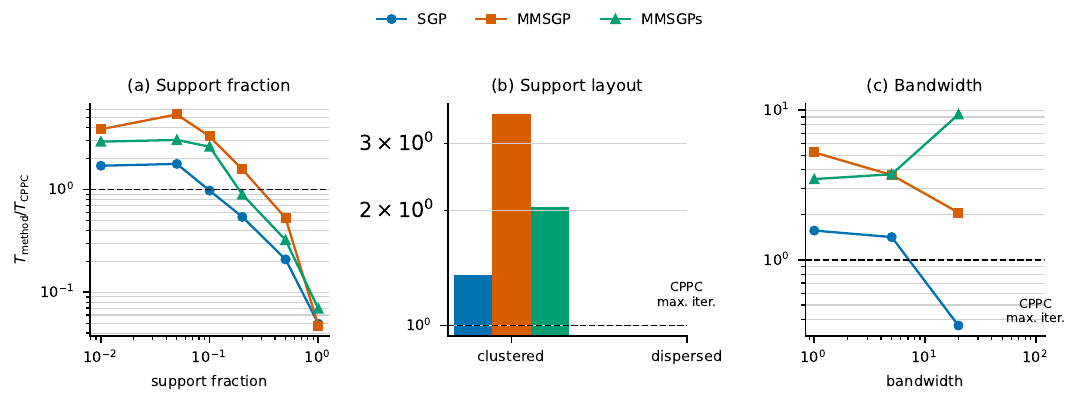}
\caption{Structural tests: runtime ratios relative to CPPC.  Missing ratios mark
CPPC runs that reached $10^4$ iterations without satisfying the tolerance. Values above one favor CPPC.}
\label{fig:structure}
\end{figure}

CPPC is fastest at support fractions $0.01$ and $0.05$ and for contiguous support,
but SGP overtakes it as the support becomes diffuse; CPPC reaches the iteration cap
for dispersed support.  It is competitive for $w\leq5$ but reaches the cap at
$w=100$.  Hence residual localization and local coupling, rather than sparsity alone,
are favorable to the coordinate predictor.

\paragraph{Matrix geometry.}
Finally, we fix $n=1000$, $w=5$, $\varepsilon=0$, and a contiguous support fraction
of $0.05$.  We vary $\kappa\in\{0,0.5,2,10\}$ with $m=0.05$, and then vary
$m\in\{0.02,0.05,0.5,2\}$ with $\kappa=0$.  Figure~\ref{fig:geometry} again
shows runtime ratios relative to CPPC.

\begin{figure}[pos=!t]
\centering
\includegraphics[width=0.8\linewidth]{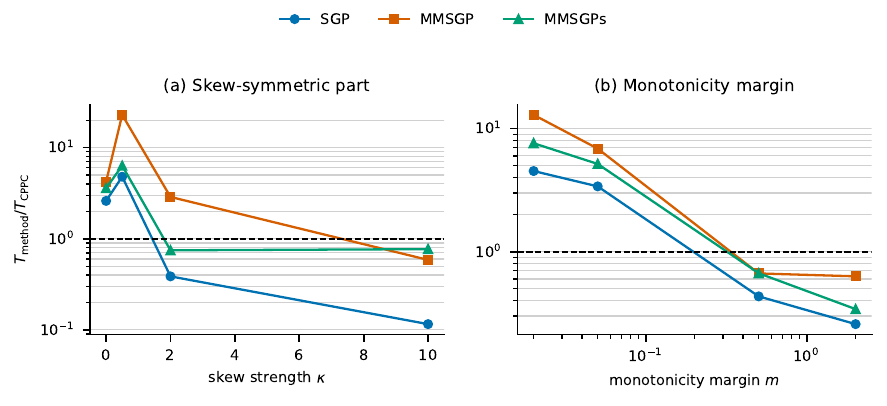}
\caption{Matrix-geometry tests: runtime ratios relative to CPPC; values above one
favor CPPC.}
\label{fig:geometry}
\end{figure}

CPPC is fastest for $\kappa=0$ and $0.5$, whereas larger skew-symmetric parts
increase the number of corrections and favor SGP.  Thus allowing nonsymmetric $A$
does not imply insensitivity to a large skew-symmetric component.  CPPC is also
fastest for $m=0.02$ and $0.05$; when $m\geq0.5$, the spectral methods converge in
few iterations and become faster.  Hence CPPC is most competitive near the boundary
of monotonicity, whereas strong monotonicity favors full-residual spectral steps.

\paragraph{SGP implementation.}
The fact that MMSGP does not generally outperform SGP in these tests requires a
clarification.  Algorithm 2.1 of \cite{Yu2009SGP} defines $\theta_{k+1}=\frac{\langle s_k,s_k\rangle}{\langle s_k,y_k\rangle}$,
but a later display in the same paper contains $\langle y_k,y_k\rangle$ in the
denominator.  In Example 2 of \cite{Yu2021MMSGP}, the latter variant reproduces the
reported SGP counts exactly, whereas the stated formula gives substantially fewer
iterations.  This exact agreement is consistent with the use of the latter denominator in the reported comparison, which results in a weaker SGP baseline.  We use the formula in Algorithm 2.1, which explains
why MMSGP need not improve upon SGP here.

Overall, CPPC is most competitive for large systems with localized residuals, local
coupling, a small monotonicity margin and a moderate skew-symmetric component; the
spectral methods can be faster outside this regime.

\section{Conclusion}
We proposed the line-search-free CPPC method for monotone AVEs.  Its exact coordinate
proximal predictor yields an algebraic acceptance certificate, avoiding backtracking and
leaving one new matrix--vector product per iteration.  Under $\sym(A)\succeq I$ and a
nonempty solution set, CPPC converges globally with an $O(K^{-1/2})$ best-iterate
residual bound; a positive
monotonicity margin further yields linear convergence.  Numerical results indicate its
advantage for large, locally coupled problems with localized residuals and a small
monotonicity margin.  Future work will investigate adaptive coordinate selection and
extensions to broader structured nonsmooth equations.

\section*{Data availability}

The code and data supporting the numerical results are publicly available at
\url{https://github.com/SKY-Tower-ST/CPPC-AVE}.

\end{document}